\documentclass[reqno]{amsart}

\usepackage{amssymb}
\usepackage{graphicx}
\usepackage{amscd}
\usepackage[hidelinks]{hyperref}
\usepackage{color}
\usepackage{float}
\usepackage{graphics,amsmath,amssymb}
\usepackage{amsthm}
\usepackage{amsfonts}
\usepackage{latexsym}
\usepackage{epsf}
\usepackage{enumerate}
\usepackage{xifthen}
\usepackage{mathrsfs}
\usepackage{dsfont}
\usepackage{makecell}
\usepackage[FIGTOPCAP]{subfigure}
\usepackage{amsmath}
\allowdisplaybreaks[4]
\usepackage{listings}
\usepackage{etoolbox}
\usepackage{fancyhdr}
\usepackage{pdflscape}
\usepackage[title,toc,titletoc]{appendix}
\usepackage{upgreek}
\usepackage[noadjust]{cite}

\usepackage[OT2, T1]{fontenc}
\usepackage[russian, english]{babel}

 \newtheoremstyle{mytheorem}
 {3pt}
 {3pt}
 {\slshape}
 {}
 {\bfseries}
 {.}
 { }
 {}

\numberwithin{equation}{section}

\theoremstyle{theorem}
\newtheorem{theorem}{Theorem}[section]

\newtheorem{lemma}[theorem]{Lemma}
\newtheorem{proposition}[theorem]{Proposition}

\theoremstyle{definition}

\newtheorem*{example*}{Example}

\theoremstyle{remark}
\newtheorem{remark}{Remark}[section]
\newtheorem*{remark*}{Remark}
\newtheorem*{remarks*}{Remarks}

\newcommand{\arxiv}[1]{\href{http://arxiv.org/abs/#1}{arXiv:#1}}

\newcommand{\Keywords}[1]{\ifthenelse{\isempty{#1}}{}{\smallskip \smallskip \noindent \textbf{Keywords}. #1}}
\newcommand{\MSC}[2][2010]{\ifthenelse{\isempty{#2}}{}{\smallskip \smallskip \noindent \textbf{#1MSC}. #2}}
\newcommand{\abstractnote}[1]{\ifthenelse{\isempty{#1}}{}{\smallskip \smallskip \noindent \textsuperscript{\dag}#1}}

\makeatletter
\def\specialsection{\@startsection{section}{1}%
  \z@{\linespacing\@plus\linespacing}{.5\linespacing}%
  {\normalfont}}
\def\section{\@startsection{section}{1}%
  \z@{.7\linespacing\@plus\linespacing}{.5\linespacing}%
  {\normalfont\scshape}}
\patchcmd{\@settitle}{\uppercasenonmath\@title}{\Large\boldmath}{}{}
\patchcmd{\@settitle}{\begin{center}}{\begin{flushleft}}{}{}
\patchcmd{\@settitle}{\end{center}}{\end{flushleft}}{}{}
\patchcmd{\@setauthors}{\MakeUppercase}{\normalsize}{}{}
\patchcmd{\@setauthors}{\centering}{\raggedright}{}{}
\patchcmd{\section}{\scshape}{\large\bfseries\boldmath}{}{}
\patchcmd{\subsection}{\bfseries}{\bfseries\boldmath}{}{}
\renewcommand{\@secnumfont}{\bfseries}
\patchcmd{\@startsection}{\@afterindenttrue}{\@afterindentfalse}{}{}
\patchcmd{\abstract}{\leftmargin3pc}{\leftmargin1pc}{}{}

\def\maketitle{\par
  \@topnum\z@ 
  \@setcopyright
  \thispagestyle{empty}
  \ifx\@empty\shortauthors \let\shortauthors\shorttitle
  \else \andify\shortauthors
  \fi
  \@maketitle@hook
  \begingroup
  \@maketitle
  \toks@\@xp{\shortauthors}\@temptokena\@xp{\shorttitle}%
  \toks4{\def\\{ \ignorespaces}}
  \edef\@tempa{%
    \@nx\markboth{\the\toks4
      \@nx\MakeUppercase{\the\toks@}}{\the\@temptokena}}%
  \@tempa
  \endgroup
  \c@footnote\z@
  \@cleartopmattertags
}
\makeatother

\newcommand{\rD}{\textup{\foreignlanguage{russian}{D}}}
\newcommand{\rZh}{\textup{\foreignlanguage{russian}{Zh}}}
\newcommand{\fC}{\mathbf{C}}
\newcommand{\fH}{\mathfrak{H}}
\newcommand{\vs}{\varsigma}
\newcommand{\vk}{\varkappa}
\newcommand{\cJ}{\mathcal{J}}
\newcommand{\cL}{\mathcal{L}}
\newcommand{\myceil}[1]{\left\lceil #1 \right\rceil}
\newcommand{\myfloor}[1]{\left\lfloor #1 \right\rfloor}
\newcommand{\bfm}{\mathbf{m}}
\newcommand{\bfr}{\mathbf{r}}
\newcommand{\bfd}{\boldsymbol{\updelta}}
\newcommand{\lcm}{\mathrm{lcm}}

\makeatletter
\newcommand{\vast}{\bBigg@{4}}
\newcommand{\Vast}{\bBigg@{5}}
\makeatother

\title{Asymptotics for the Taylor coefficients of certain infinite products}

\author[S. Chern]{Shane Chern}
\address{Department of Mathematics, The Pennsylvania State University, University Park, PA 16802, USA}
\email{shanechern@psu.edu}

\date{}

\begin{document}

{\footnotesize\noindent \textit{Preprint} (2019). Available at \arxiv{1902.10839}.}

\bigskip \bigskip

\maketitle

\begin{abstract}

Let $(m_1,\ldots,m_J)$ and $(r_1,\ldots,r_J)$ be two sequences of $J$ positive integers satisfying $1\le r_j< m_j$ for all $j=1,\ldots,J$. Let $(\delta_1,\ldots,\delta_J)$ be a sequence of $J$ nonzero integers. In this paper, we study the asymptotic behavior of the Taylor coefficients of the infinite product
$$\prod_{j=1}^J\Bigg(\prod_{k\ge 1}\big(1-q^{r_j+m_j(k-1)}\big)\big(1-q^{-r_j+m_jk}\big)\Bigg)^{\delta_j}.$$


\MSC{11P55.}
\end{abstract}


\section{Introduction}

\subsection{Motivations}

For complex variables $\alpha$ and $q$ with $|q|<1$, we denote
$$(\alpha;q)_n:=\prod_{k= 0}^{n-1}(1-\alpha q^k)\quad\text{and}\quad(\alpha;q)_\infty:=\prod_{k\ge 0}(1-\alpha q^k).$$
We also use the notation
$$(\alpha,\beta,\ldots,\gamma;q)_\infty:=(\alpha;q)_\infty(\beta;q)_\infty\cdots(\gamma;q)_\infty.$$

\medskip

Let $p(n)$ be the number of partitions of $n$; that is, the number of representations of $n$ written as a sum of a non-increasing sequence of positive integers. It is well known that $p(n)$ has the generating function
$$\sum_{n\ge 0}p(n)q^n=\frac{1}{(q;q)_\infty}.$$
The study of the asymptotic behavior of $p(n)$ originates from Hardy and Ramanujan \cite{HR1917}. A couple of decades later, Rademacher \cite{Rad1937} further proved the following formula
\begin{align}
p(n)=\frac{1}{2\sqrt{2} \pi}\sum_{k\ge 1}A_k(n)\sqrt{k}\;\frac{d}{dn}\left(\frac{2}{\sqrt{n-\frac{1}{24}}}\sinh\left(\frac{\pi}{k}\sqrt{\frac{2}{3}\left(n-\frac{1}{24}\right)}\right)\right),
\end{align}
where
$$A_k(n)=\sum_{\substack{0\le h< k\\ \gcd(h,k)=1}}e^{\pi i(s(h,k)-2nh/k)}$$
with $s(h,k)$ being the Dedekind sum defined in \eqref{eq:dedekind-sum}.

Apart from ordinary partitions, partitions under symmetric congruence conditions also attract broad research interest. The most famous examples arise from the Rogers--Ramanujan identities (Rogers \cite{Rog1894}, Ramanujan \cite{Ram1919}). Here the first Rogers--Ramanujan identity states that (cf.~Corollary 7.67 in \cite{And1976})
$$\frac{1}{(q,q^4;q^5)_\infty}=\sum_{n\ge 0}\frac{q^{n^2}}{(q;q)_n}.$$
Using the language in partition theory, the above identity can be restated as follows. The number of partitions of $n$ such that each part is congruent to $\pm 1$ modulo $5$ equals the number of partitions of $n$ such that the adjacent parts differ by at least two. Let $p_{5,\pm 1}(n)$ be the number of partitions of $n$ such that each part is congruent to $\pm 1$ modulo $5$. Its asymptotic formula was shown by Lehner \cite{Leh1941}:
\begin{align}\label{eq:51}
p_{5,\pm 1}(n)\sim \frac{\csc(\pi/5)}{4\cdot 3^{1/4}\cdot 5^{1/4}}n^{-3/4}\exp\Bigg(2\pi\sqrt{\frac{n}{15}}\Bigg).
\end{align}
The interested reader may also refer to Niven \cite{Niv1940}, Livingood \cite{Liv1945}, Petersson \cite{Pet1954,Pet1956}, Grosswald \cite{Gro1958}, Iseki \cite{Ise1959,Ise1960,Ise1961}, Hagis Jr.~\cite{Hag1962}, Subrahmanyasastri \cite{Sub1972} and so forth for the asymptotic behaviors of other partition functions under symmetric congruence conditions.

As we have seen, the generating function of $p_{5,\pm 1}(n)$ is indeed an infinite product under a symmetric congruence condition. Further, similar infinite products are also of number-theoretic interest. One example is the Rogers--Ramanujan continued fraction defined by
$$R(q) := \cfrac{q^{1/5}}{1 + \cfrac{q}{1 + \cfrac{q^{2}}{1 + \cfrac{q^{3}}{1 + \cdots}}}}=q^{1/5}\frac{(q,q^4;q^5)_\infty}{(q^2,q^3;q^5)_\infty}.$$
Let us focus on the infinite product part in $R(q)$ and write
$$\sum_{n\ge 0}C(n)q^n=\frac{(q,q^4;q^5)_\infty}{(q^2,q^3;q^5)_\infty}.$$
It is known from Richmond and Szekeres \cite{RS1978} that
\begin{align}\label{eq:C(n)}
C(n)\sim \frac{2^{1/2}}{5^{3/4}}\cos\Bigg(\frac{4\pi}{5}\bigg(n+\frac{3}{20}\bigg)\Bigg)n^{-3/4}\exp\Bigg(\frac{4\pi}{5}\sqrt{\frac{n}{5}}\Bigg).
\end{align}
Hence for sufficiently large $n$, $C(5n+0,2)>0$ and $C(5n+1,3,4)<0$. We also remark that in \cite{RS1978}, Richmond and Szekeres indeed studied the asymptotic behavior of the Taylor coefficients of the general infinite product
$$\prod_{j=1}^{m-1}(q^j;q^m)^{-\zeta \chi(j)}$$
where $m$ is a positive fundamental discriminant, $\chi(j)=(m | j)$ is the Kronecker symbol and $\zeta$ is either $1$ or $-1$.

Recently, there are a number of papers \cite{AG1994,AB1979,Hir2018,McL2015,Tan2018} studying vanishing Taylor coefficients of certain infinite products. For instance, Tang \cite{Tan2018} showed that the Taylor coefficients of
$$\sum_{n\ge 0} B(n)q^n=(-q^2,-q^3;q^5)_\infty^2(q^2,q^8;q^{10})_\infty=\frac{(q^2,q^8;q^{10})_\infty(q^4;q^6;q^{10})_\infty^2}{(q^2,q^3;q^5)_\infty^2}$$
satisfy $B(5n+1)=0$ for all $n\ge 0$. At the end of Tang's paper, he also provided numerical evidence of the inequalities $B(5n+0,2,3)>0$ and $B(5n+4)<0$ for sufficiently large $n$. Similar numerical evidences are also provided for inequalities of Taylor coefficients of other infinite products.

\medskip

Motivated by these work, it is natural to investigate a broad family of infinite products. Let $\bfm=(m_1,\ldots,m_J)$ and $\bfr=(r_1,\ldots,r_J)$ be two sequences of $J$ positive integers satisfying $1\le r_j< m_j$ for all $j=1,\ldots,J$. Let $\bfd=(\delta_1,\ldots,\delta_J)$ be a sequence of $J$ nonzero integers. In this paper, we shall study the asymptotics for the Taylor coefficients of the following infinite product
\begin{equation}
\sum_{n\ge 0}g(n)q^n=\prod_{j=1}^J (q^{r_j},q^{m_j-r_j};q^{m_j})_\infty^{\delta_j}.
\end{equation}

\subsection{Notation and main result}

Let $\fC$ be the set of complex numbers and $\fH$ be the upper half complex plane. Let $\gcd$ and $\lcm$ be the greatest common divisor function and least common multiple function, respectively. For a positive integer $n$, we accept the convention that $\gcd(0,n)=n$.

We define the big-$O$ notation as usual: $f(x)=O(g(x))$ means that $|f(x)|\le C g(x)$ where $C$ is an absolute constant. Furthermore, $f(x)\ll g(x)$ means that $f(x)=O(g(x))$. Throughout this paper, we always assume that the constant $C$ depends on $\bfm$, $\bfr$ and $\bfd$ unless otherwise stated.

\medskip

Below we assume that $0\le h<k$ are integers such that $\gcd(h,k)=1$. Let us define auxiliary functions
$$\lambda_{m,r}(h,k):=\myceil{\frac{rh}{\gcd(m,k)}}$$
and
$$\lambda^*_{m,r}(h,k):=\lambda_{m,r}(h,k)-\frac{rh}{\gcd(m,k)}.$$
We also put $\hbar_m(h,k)$ an integer such that
$$\hbar_m(h,k) \frac{m h}{\gcd(m,k)}\equiv -1 \pmod{\frac{k}{\gcd(m,k)}}.$$
Notice that one may always find such an integer since $\gcd(h,k)=1$.

\medskip

Next, we define
$$\Omega:=\sum_{j=1}^J\delta_j \Bigg(2m_j-12r_j+\frac{12r_j^2}{m_j}\Bigg),$$
$$\Delta(h,k):=-\sum_{j=1}^J\delta_j \Bigg(\frac{2\gcd^2(m_j,k)}{m_j}+\frac{12\gcd^2(m_j,k)}{m_j}({\lambda^*}^2_{\!\!\!m_j,r_j}(h,k)-\lambda^*_{m_j,r_j}(h,k))\Bigg)$$
and
\begin{align}
\omega_{h,k}:=\exp\left(-\pi i\sum_{j=1}^J \delta_j\cdot s\left(\frac{m_j h}{\gcd(m_j,k)},\frac{k}{\gcd(m_j, k)}\right)\right),
\end{align}
where $s(d,c)$ is the Dedekind sum defined by
\begin{equation}\label{eq:dedekind-sum}
s(d,c):=\sum_{n \bmod{c}} \bigg(\!\!\bigg(\frac{dn}{c}\bigg)\!\!\bigg)\bigg(\!\!\bigg(\frac{n}{c}\bigg)\!\!\bigg)
\end{equation}
with
$$(\!(x)\!):=\begin{cases}
x-\lfloor x\rfloor -1/2 & \text{if $x\not\in \mathbb{Z}$},\\
0 & \text{if $x\in \mathbb{Z}$}.
\end{cases}$$
We also define
\begin{align*}
\rD_{h,k}&:=\exp\Bigg(\pi i \sum_{j=1}^J \delta_j\bigg(\frac{r_j h}{k}-\frac{r_j \gcd(m_j,k)}{m_j k}+\frac{2r_j \gcd(m_j,k)\lambda^*_{m_j,r_j}(h,k)}{m_j k}\\
&\ \quad+\frac{\hbar_{m_j}(h,k)\gcd(m_j,k)}{k}(\lambda^2_{m_j,r_j}(h,k)-\lambda_{m_j,r_j}(h,k))\bigg)\Bigg).
\end{align*}
One readily verifies that the choice of $\hbar_m(h,k)$ does not affect the value of $\rD_{h,k}$. At last, we define
\begin{align*}
\Pi_{h,k}:=\begin{cases}
\displaystyle\prod_{j:\lambda^*_{m_j,r_j}(h,k)=0} \Bigg(1-\exp\bigg(2\pi i\frac{r_{j}\gcd(m_{j},k)+r_{j}\hbar_{m_{j}}(h,k)m_{j}h}{m_{j}k}\bigg)\Bigg)^{\delta_j} \\\hfill\text{{if there exists $j$ such that $\lambda^*_{m_j,r_j}(h,k)=0$}},\\
\quad\quad\ \ 1 \\\hfill\text{otherwise}.
\end{cases}
\end{align*}
Remark \ref{rmk:Pi-value} tells us that the choice of $\hbar_m(h,k)$ also does not affect the value of $\Pi_{h,k}$. Also, Proposition \ref{lem:pre} indicates that for any $j$ with $\lambda^*_{m_j,r_j}(h,k)=0$, we have
$$1-\exp\bigg(2\pi i\frac{r_{j}\gcd(m_{j},k)+r_{j}\hbar_{m_{j}}(h,k)m_{j}h}{m_{j}k}\bigg) \ne 0.$$
Hence the value $\Pi_{h,k}$ is well-defined and $\Pi_{h,k}\ne 0$.

\medskip

Given a real $0\le x<1$, we define
$$\Upsilon(x):=\begin{cases}
1 & \text{if $x=0$},\\
x & \text{if $0<x\le 1/2$},\\
1-x & \text{if $1/2<x<1$}.
\end{cases}$$

\medskip

Let $L=\lcm(m_1,\ldots,m_R)$. We define two disjoint sets:
\begin{align*}
\cL_{>0}&:=\{(\vk,\ell) \; :\; 1\le \ell\le L,\; 0\le \vk<\ell,\; \Delta(\vk,\ell)>0\},\\
\cL_{\le 0}&:=\{(\vk,\ell) \; :\; 1\le \ell\le L,\; 0\le \vk<\ell,\; \Delta(\vk,\ell)\le 0\}.
\end{align*}

\medskip

Our main result states as follows.

\begin{theorem}
If the inequality
\begin{equation}\label{eq:assump}
\min_{1\le j\le J}\left(\Upsilon\big(\lambda^*_{m_j,r_j}(\vk,\ell)\big)\frac{\gcd^2(m_j,\ell)}{m_j}\right)\ge \frac{\Delta(\vk,\ell)}{24}
\end{equation}
holds for all $1\le \ell \le L$ and $0\le \vk<\ell$, then for positive integers $n>-\Omega/24$, we have
\begin{align}
g(n)&=E(n)+2\pi i^{\sum_{j=1}^{J}\delta_j}\underset{(\vk,\ell)\in\cL_{>0}}{\sum_{1\le \ell\le L}\sum_{0\le \vk< \ell}} \left(\frac{24n+\Omega}{\Delta(\vk,\ell)}\right)^{-\frac{1}{2}}\notag\\
&\quad\quad\quad\quad\quad\quad\quad\quad\quad\times \sum_{\substack{1\le k\le N^\star\\k \equiv \ell \bmod{L}}} \frac{1}{k}  I_{-1}\left(\frac{\pi }{6k}\sqrt{\Delta(\vk,\ell)(24n+\Omega)}\right)\notag\\
&\quad\quad\quad\quad\quad\quad\quad\quad\quad\times \sum_{\substack{0\le h< k\\ \gcd(h,k)=1\\ h\equiv \vk \bmod{\ell}}} e^{-\frac{2\pi i nh}{k}}(-1)^{\sum_{j=1}^{J}\delta_j \lambda_{m_j,r_j}(h,k)}\omega_{h,k}^2\;\rD_{h,k}\Pi_{h,k},\label{eq:main-result}
\end{align}
where
\begin{equation}
N^\star=\myfloor{\sqrt{2\pi \left(n+\frac{\Omega}{24}\right)}},
\end{equation}
$I_s(x)$ is the modified Bessel function of the first kind, and
\begin{equation}
E(n)\ll_{\bfm,\bfr,\bfd} 1.
\end{equation}
\end{theorem}

\begin{remark}
To better understand the asymptotic behavior of $g(n)$, one may apply the asymptotic expansion of $I_s(x)$ (cf.~\cite[p.~377, (9.7.1)]{AA1972}): for fixed $s$, when $|\arg x|<\frac{\pi}{2}$,
\begin{align}\label{Bessel-order}
I_{s}(x)\sim \frac{e^x}{\sqrt{2\pi x}}\left(1-\frac{4s^2-1}{8x}+\frac{(4s^2-1)(4s^2-9)}{2!(8x)^2}-\cdots \right).
\end{align}
\end{remark}

\section{Applications of the main result}

Before moving to the proof of the main result, we first give some applications. In the first two examples, we reproduce the asymptotic formulas \eqref{eq:51} and \eqref{eq:C(n)}, respectively. We then confirm Tang's inequalities in \cite{Tan2018} in the asymptotic sense. In this section, we always expand the infinite product as $\sum_{n\ge 0}g(n)q^n$.

In general, to obtain an explicit asymptotic formula of $g(n)$, we first compute $\cL_{>0}$. Next, we find the largest number among $\{\sqrt{\Delta(\vk,\ell)}/k\}$ with $(\vk,\ell)\in\cL_{>0}$ and $k\equiv \ell \pmod{L}$. Now one needs to check if the corresponding $I$-Bessel function vanishes for this choice. If it is nonvanishing, then the asymptotic formula shall be obtained from the $I$-Bessel term. Otherwise, we move to find the second largest number among $\{\sqrt{\Delta(\vk,\ell)}/k\}$ and carry out the same program. Notice that if there are multiple choices of $\vk$, $\ell$ and $k$ giving the same value of $\sqrt{\Delta(\vk,\ell)}/k$, one should sum up all such $I$-Bessel terms and check if the summation vanishes or not.

\subsection{Partitions into parts congruent to $\pm 1$ modulo $5$}

Let
$$\sum_{n\ge 0}g(n)q^n=\frac{1}{(q,q^4;q^5)_\infty}.$$
Then $\bfm=\{5\}$, $\bfr=\{1\}$ and $\bfd=\{-1\}$. Hence $L=5$ and $\Omega=-2/5$. We now compute that
\begin{align*}
\cL_{>0}&=\{(0, 1), (0, 2), (1, 2), (0, 3), (1, 3), (2, 3),\\
&\quad\;\;\; (0, 4), (1, 4), (2, 4), (3, 4), (0, 5), (1, 5), (4, 5)\}.
\end{align*}
First, the assumption \eqref{eq:assump} is satisfied. We next find that the largest number among $\{\sqrt{\Delta(\vk,\ell)}/k\}$ with $(\vk,\ell)\in\cL_{>0}$ and $k\equiv \ell \pmod{L}$ is $\sqrt{\frac{2}{5}}$. Here we have two choices:
$$(\vk,\ell,k)=(0,1,1),\ (0,5,5).$$
When $k=1$, the admissible $(h,k)$ is $(0,1)$. We compute that the $I$-Bessel term is
$$\frac{\pi \csc\big(\frac{\pi}{5}\big)}{2\sqrt{15}}\Bigg(n-\frac{1}{60}\Bigg)^{-1/2}I_{-1}\Bigg(\frac{2\pi}{\sqrt{15}}\sqrt{n-\frac{1}{60}}\Bigg).$$
When $k=5$, noticing that $\gcd(0,5)=5\ne 1$, there is no admissible $(h,k)$. Hence,
\begin{align*}
g(n)&\sim \frac{\pi \csc\big(\frac{\pi}{5}\big)}{2\sqrt{15}}\Bigg(n-\frac{1}{60}\Bigg)^{-1/2}I_{-1}\Bigg(\frac{2\pi}{\sqrt{15}}\sqrt{n-\frac{1}{60}}\Bigg)\\
&\sim \frac{\csc\big(\frac{\pi}{5}\big)}{4\cdot 3^{1/4}\cdot 5^{1/4}}\Bigg(n-\frac{1}{60}\Bigg)^{-3/4}\exp\Bigg(\frac{2\pi}{\sqrt{15}}\sqrt{n-\frac{1}{60}}\Bigg).
\end{align*}

\subsection{The Rogers--Ramanujan continued fraction}

Let
$$\sum_{n\ge 0}g(n)q^n=\frac{(q,q^4;q^5)_\infty}{(q^2,q^3;q^5)_\infty}.$$
Then $\bfm=\{5,5\}$, $\bfr=\{1,2\}$ and $\bfd=\{1,-1\}$. Hence $L=5$ and $\Omega=24/5$. We compute that
\begin{align*}
\cL_{>0}&=\{(2, 5),(3, 5)\}.
\end{align*}
First, the assumption \eqref{eq:assump} is satisfied. We next find that the largest number among $\{\sqrt{\Delta(\vk,\ell)}/k\}$ with $(\vk,\ell)\in\cL_{>0}$ and $k\equiv \ell \pmod{L}$ is $\frac{2\sqrt{6}}{5\sqrt{5}}$. Here we have two choices:
$$(\vk,\ell,k)=(2, 5,5),\ (3, 5,5).$$
When $k=5$, the admissible $(h,k)$ are $(2,5)$ and $(3,5)$. We compute that, in total, the $I$-Bessel term is
$$\frac{4\pi}{5\sqrt{5}}\cos\Bigg(\frac{4\pi}{5}\bigg(n+\frac{3}{20}\bigg)\Bigg)\Bigg(n+\frac{1}{5}\Bigg)^{-1/2}I_{-1}\Bigg(\frac{4\pi}{5\sqrt{5}}\sqrt{n+\frac{1}{5}}\Bigg).$$
Notice that $\cos\big(\frac{4\pi}{5}(n+\frac{3}{20})\big)$ does not vanish for all $n$. Hence,
\begin{align*}
g(n)&\sim\frac{4\pi}{5\sqrt{5}}\cos\Bigg(\frac{4\pi}{5}\bigg(n+\frac{3}{20}\bigg)\Bigg)\Bigg(n+\frac{1}{5}\Bigg)^{-1/2}I_{-1}\Bigg(\frac{4\pi}{5\sqrt{5}}\sqrt{n+\frac{1}{5}}\Bigg)\\
&\sim \frac{2^{1/2}}{5^{3/4}}\cos\Bigg(\frac{4\pi}{5}\bigg(n+\frac{3}{20}\bigg)\Bigg)\Bigg(n+\frac{1}{5}\Bigg)^{-3/4}\exp\Bigg(\frac{4\pi}{5\sqrt{5}}\sqrt{n+\frac{1}{5}}\Bigg).
\end{align*}

\subsection{Tang's inequalities}

Let
$$\sum_{n\ge 0}g(n)q^n=\frac{(q^2,q^8;q^{10})_\infty(q^4;q^6;q^{10})_\infty^2}{(q^2,q^3;q^5)_\infty^2}.$$
Then $\bfm=\{5,10,10\}$, $\bfr=\{2,2,4\}$ and $\bfd=\{-2,1,2\}$. Hence $L=10$ and $\Omega=-8$. We compute that
\begin{align*}
\cL_{>0}&=\{(0, 1), (0, 3), (1, 3), (2, 3), (0, 5), (2, 5), (3, 5), (0, 7), (1, 7), (2, 7), (3, 7),\\
&\quad\;\;\; (4, 7), (5, 7), (6, 7), (0, 9), (1, 9), (2, 9), (3, 9), (4, 9), (5, 9), (6, 9), (7, 9),\\
&\quad\;\;\; (8, 9), (1, 10), (2, 10), (3, 10), (4, 10), (6, 10), (7, 10), (8, 10), (9, 10)\}.
\end{align*}
First, the assumption \eqref{eq:assump} is satisfied. We next find that the largest number among $\{\sqrt{\Delta(\vk,\ell)}/k\}$ with $(\vk,\ell)\in\cL_{>0}$ and $k\equiv \ell \pmod{L}$ is $\frac{1}{\sqrt{5}}$. Here we have four choices:
$$(\vk,\ell,k)=(0,1,1),\ (0,5,5),\ (2, 5,5),\ (3, 5,5).$$
When $k=1$, the admissible $(h,k)$ is $(0,1)$. We compute that the $I$-Bessel term is
$$\frac{\sqrt{2}\pi}{\sqrt{15}}\sin\Bigg(\frac{\pi}{5}\Bigg)\Bigg(n-\frac{1}{3}\Bigg)^{-1/2}I_{-1}\Bigg(\frac{\sqrt{2}\pi}{\sqrt{15}}\sqrt{n-\frac{1}{3}}\Bigg).$$
When $k=5$, the admissible $(h,k)$ are $(2,5)$ and $(3,5)$. We compute that, in total, the $I$-Bessel term is
$$\frac{\sqrt{2}\pi}{\sqrt{15}}\sin\Bigg(\frac{2\pi}{5}\big(2n+1\big)\Bigg)\Bigg(n-\frac{1}{3}\Bigg)^{-1/2}I_{-1}\Bigg(\frac{\sqrt{2}\pi}{\sqrt{15}}\sqrt{n-\frac{1}{3}}\Bigg).$$
In total, we therefore have
$$\frac{\sqrt{2}\pi}{\sqrt{15}}\Bigg(\sin\bigg(\frac{\pi}{5}\bigg)+\sin\bigg(\frac{2\pi}{5}\big(2n+1\big)\bigg)\Bigg)\Bigg(n-\frac{1}{3}\Bigg)^{-1/2}I_{-1}\Bigg(\frac{\sqrt{2}\pi}{\sqrt{15}}\sqrt{n-\frac{1}{3}}\Bigg).$$
Notice that $\sin\big(\frac{\pi}{5}\big)+\sin\big(\frac{2\pi}{5}(2n+1)\big)$ vanishes only if $n\equiv 1 \pmod{5}$. Hence, for $n\not\equiv 1 \pmod{5}$,
\begin{align*}
g(n)&\sim\frac{\sqrt{2}\pi}{\sqrt{15}}\Bigg(\sin\bigg(\frac{\pi}{5}\bigg)+\sin\bigg(\frac{2\pi}{5}\big(2n+1\big)\bigg)\Bigg)\Bigg(n-\frac{1}{3}\Bigg)^{-1/2}I_{-1}\Bigg(\frac{\sqrt{2}\pi}{\sqrt{15}}\sqrt{n-\frac{1}{3}}\Bigg)\\
&\sim \frac{1}{30^{1/4}}\Bigg(\sin\bigg(\frac{\pi}{5}\bigg)+\sin\bigg(\frac{2\pi}{5}\big(2n+1\big)\bigg)\Bigg)\Bigg(n-\frac{1}{3}\Bigg)^{-3/4}\exp\Bigg(\frac{\sqrt{2}\pi}{\sqrt{15}}\sqrt{n-\frac{1}{3}}\Bigg).
\end{align*}
It follows that $g(5n+0,2,3)>0$ and $g(5n+4)<0$ for sufficiently large $n$. If we further compute a number of lower $I$-Bessel terms, we still encounter the same vanishment for $n\equiv 1 \pmod{5}$. This highly suggests that $g(5n+1)=0$, which is, indeed, proved by Tang using elementary techniques in \cite{Tan2018}.

\medskip

All other inequalities conjectured by Tang can be proved for sufficiently large $n$ in the same manner. We omit the details here.

\section{Dedekind eta function and Jacobi theta function}

In this section, we introduce the Dedekind eta function and Jacobi theta function. All results here are standard, which can be found in, for example, \cite{Apo1990} or \cite{Zwe2002}.

Let $\tau\in\fH$ and $\vs\in\fC$. The Dedekind eta function is defined by
$$\eta(\tau):=q^{1/24}(q;q)_\infty$$
with $q:=e^{2\pi i \tau}$. Further, the Jacobi theta function reads
$$\vartheta(\vs;\tau):=\sum_{\nu\in\mathbb{Z}+\frac{1}{2}}e^{2\pi i \nu (\vs+\frac{1}{2})+\pi i\nu^2 \tau}.$$
Notice that if we put $\zeta:=e^{2\pi i \vs}$, then the Jacobi triple product identity indicates that
$$\vartheta(\vs;\tau)=-i q^{1/8}\zeta^{-1/2}(\zeta,\zeta^{-1}q,q;q)_\infty.$$
It follows immediately that
\begin{align}
\rZh(\vs;\tau):=&\; (\zeta,\zeta^{-1}q;q)_\infty\notag\\[0.5em]
=&\; ie^{-\frac{\pi i \tau}{6}}e^{\pi i \vs}\frac{\vartheta(\vs;\tau)}{\eta(\tau)}.\label{eq:Zh-basic}
\end{align}

The Dedekind eta function and Jacobi theta function are of broad interest due to their transformation properties. Let $\displaystyle\gamma=\begin{pmatrix}a&b\\c&d\end{pmatrix}\in SL_2(\mathbb{Z})$ where we assume that $c>0$. Recall that the M\"obius transformation for $\tau\in\fH$ is defined by
$$\gamma(\tau):=\frac{a\tau +b}{c\tau + d}.$$
Further, for the $\gamma$ given above, we write for convenience
$$\gamma^*(\tau):=\frac{1}{c\tau + d}.$$
If
$$\chi(\gamma)=\exp\Bigg(\pi i \left(\frac{a+d}{12c}-s(d,c)-\frac{1}{4}\right)\Bigg),$$
where, again, $s(d,c)$ is the Dedekind sum, then
\begin{align}\label{eta-trans}
\eta(\gamma(\tau))=\chi(\gamma)(c\tau+d)^{1/2}\eta(\tau)
\end{align}
and
\begin{align}\label{theta-trans1}
\vartheta(\vs\gamma^*(\tau);\gamma(\tau))=\chi(\gamma)^3(c\tau+d)^{1/2}e^{\frac{\pi i c \vs^2}{c\tau+d}}\vartheta(\vs;\tau).
\end{align}
Further, let $\alpha$ and $\beta$ be integers. The Jacobi theta function also satisfies
\begin{align}\label{theta-trans2}
\vartheta(\vs+\alpha\tau+\beta;\tau)=(-1)^{\alpha+\beta}e^{-\pi i \alpha^2 \tau}e^{-2\pi i\alpha \vs}\vartheta(\vs;\tau).
\end{align}

\section{Farey arcs and a transformation formula}\label{sec:Farey}

To study the asymptotics for the Taylor coefficients of $G(q)$, we turn to the celebrated circle method due to Rademacher \cite{Rad1937,Rad1943} whose idea originates from Hardy and Ramanujan \cite{HR1917}. Recalling that $G(q)$ is holomorphic inside the unit disk, we may directly apply Cauchy's integral formula to deduce
\begin{equation*}
g(n)=\frac{1}{2 \pi i} \oint_{\mathcal{C}:|q|=r} \frac{G(q)}{q^{n+1}}\ dq,
\end{equation*}
where the contour integral is taken counter-clockwise. Now one puts $r=e^{-2 \pi \varrho}$ with $\varrho=1/N^2$ where $N$ is a sufficiently large positive integer.

Next, we dissect the circle $\mathcal{C}$ by Farey arcs. Let $h/k$ with $\gcd(h,k)=1$ be a Farey fraction of order $N$. If we denote by $\xi_{h,k}$ the interval $[-\theta'_{h,k},\theta''_{h,k}]$ with $\theta'_{h,k}$ and $\theta''_{h,k}$ being the positive distances from $h/k$ to its neighboring mediants, then
\begin{equation*}
g(n)=\sum_{1\le k\le N} \sum_{\substack{0\le h< k\\ \gcd(h,k)=1}} e^{-\frac{2\pi i nh}{k}} \int_{\xi_{h,k}} G\big(e^{2\pi i (h/k+i \varrho +\phi)}\big) e^{-2\pi i n \phi} e^{2 \pi n \varrho}\ d\phi.
\end{equation*}
Making the changes of variables $z=k(\varrho -i \phi)$ and $\tau = (h+i z)/k$ yields
\begin{equation}\label{eq:cauchy-var}
g(n)=\sum_{1\le k\le N} \sum_{\substack{0\le h< k\\ \gcd(h,k)=1}} e^{-\frac{2\pi i nh}{k}} \int_{\xi_{h,k}} G\big(e^{2\pi i \tau}\big)e^{-2\pi i n \phi} e^{2 \pi n \varrho}\ d\phi.
\end{equation}

\medskip

Let $r<m$ be positive integers. Our next task is to apply the transformation properties of the Dedekind eta function and Jacobi theta function so that $\rZh(r\tau;m\tau)$ can be nicely reformulated around the Farey arc with respect to $h/k$. To do so, we need to construct a suitable matrix in $SL_2(\mathbb{Z})$.

Let $d=\gcd(m,k)$. For convenience, we write $m=d m'$ and $k= d k'$. Recalling that $\hbar_m(h,k)$ satisfies $\hbar_m(h,k) m' h\equiv -1 \pmod{k'}$, we put $b_{m'}=(\hbar_m(h,k) m' h+1)/k'$. It is straightforward to verify that the following matrix is in $SL_2(\mathbb{Z})$:
\begin{equation}
\gamma_{(m,h,k)}=\begin{pmatrix}\hbar_m(h,k) & -b_{m'}\\k' & -m' h\end{pmatrix}.
\end{equation}

\medskip

Since $\tau = (h+i z)/k = (h+i z)/dk'$, one may compute
\begin{align*}
&\gamma_{(m,h,k)}(m\tau)\\
&\quad=\dfrac{\hbar_m(h,k)\cdot m\frac{h+iz}{d k'}-b_{m'}}{k' \cdot m\frac{h+iz}{d k'}-m' h}=\frac{\hbar_m(h,k) m' h+\tilde{h}_{m'}(im'z)-(\hbar_m(h,k)m' h+1)}{m' h k'+k'(im'z)-m' hk'}\\
&\quad=\frac{\hbar_m(h,k)}{k'}+\frac{1}{m'k'z}i.
\end{align*}
Namely,
\begin{equation}\label{eq:gamma}
\gamma_{(m,h,k)}(m\tau)=\frac{\hbar_m(h,k) \gcd(m,k)}{k}+\frac{\gcd^2(m,k)}{mkz}i.
\end{equation}
On the other hand, we have
\begin{align*}
\gamma^*_{(m,h,k)}(m\tau)=\dfrac{1}{k' \cdot m\frac{h+iz}{d k'}-m' h}=-\frac{\gcd(m,k)}{mz}i
\end{align*}
and hence
\begin{equation}\label{eq:gamma-r}
r\tau\gamma^*_{(m,h,k)}(m\tau)=\frac{r \gcd(m,k)}{mk}-\frac{rh\gcd(m,k)}{mkz}i.
\end{equation}
Further,
\begin{align}
&r\tau\gamma^*_{(m,h,k)}(m\tau)+\lambda_{m,r}(h,k)\gamma_{(m,h,k)}(m\tau)\notag\\
&\quad=\frac{r \gcd(m,k)}{mk}+\lambda_{m,r}(h,k)\frac{\hbar_m(h,k) \gcd(m,k)}{k}+\lambda^*_{m,r}(h,k)\frac{\gcd^2(m,k)}{mkz}i.\label{eq:gamma-mix}
\end{align}

\medskip

Recalling from \eqref{eq:Zh-basic} that
\begin{align*}
\rZh(r\tau;m\tau)=ie^{-\frac{\pi i m\tau}{6}}e^{\pi i r\tau}\frac{\vartheta(r\tau;m\tau)}{\eta(m\tau)},
\end{align*}
one has, from \eqref{eta-trans}, \eqref{theta-trans1}, \eqref{theta-trans2} and the fact $s(-m'h,k')=-s(m'h,k')$, that
\begin{align*}
\rZh(r\tau;m\tau)&=ie^{-\frac{\pi i m\tau}{6}}e^{\pi i r\tau} \chi(\gamma_{(m,h,k)})^{-2}e^{-\frac{\pi i k' r^2\tau^2}{k'm\tau-m' h}}\\
&\quad\times\frac{\vartheta(r\tau\gamma^*_{(m,h,k)}(m\tau);\gamma_{(m,h,k)}(m\tau))}{\eta(\gamma_{(m,h,k)}(m\tau))}\\
&=ie^{-\frac{\pi i m\tau}{6}}e^{\pi i r\tau} \chi(\gamma_{(m,h,k)})^{-2}e^{-\frac{\pi i k r^2\tau^2}{km\tau-m h}}(-1)^{\lambda_{m,r}(h,k)}\\
&\quad\times e^{\pi i \lambda^2_{m,r}(h,k)\gamma_{(m,h,k)}(m\tau)}e^{2\pi i \lambda_{m,r}(h,k)r\tau \gamma^*_{(m,h,k)}(m\tau)}\\
&\quad\times \frac{\vartheta(r\tau\gamma^*_{(m,h,k)}(m\tau)+\lambda_{m,r}(h,k)\gamma_{(m,h,k)}(m\tau);\gamma_{(m,h,k)}(m\tau))}{\eta(\gamma_{(m,h,k)}(m\tau))}\\
&=i (-1)^{\lambda_{m,r}(h,k)}e^{-2\pi i s(m'h,k')}\\
&\quad\times \exp\Bigg(\pi i\bigg(\frac{rh}{k}-\frac{rd}{mk}+\frac{2rd\lambda^*_{m,r}(h,k)}{mk}\\
&\qquad\qquad\quad+\frac{\hbar_m(h,k)d}{k}(\lambda^2_{m,r}(h,k)-\lambda_{m,r}(h,k))\bigg)\Bigg)\\
&\quad\times \exp\Bigg(\frac{\pi}{12k}\bigg(\Big(2m-12r+\frac{12r^2}{m}\Big)z\\
&\qquad\qquad\quad-\Big(\frac{2d^2}{m}+\frac{12d^2}{m}({\lambda^*}^2_{\!\!\!m,r}(h,k)-\lambda^*_{m,r}(h,k))\Big)\frac{1}{z}\bigg)\Bigg)\\
&\quad\times \rZh\big(r\tau\gamma^*_{(m,h,k)}(m\tau)+\lambda_{m,r}(h,k)\gamma_{(m,h,k)}(m\tau);\gamma_{(m,h,k)}(m\tau)\big).
\end{align*}
Consequently, we deduce the following transformation formula.

\begin{align}
&G(e^{2\pi i \tau})=\prod_{j=1}^{J}\rZh^{\delta_j}(r_j\tau;m_j\tau)\notag\\
&\quad=i^{\sum_{j=1}^{J}\delta_j}(-1)^{\sum_{j=1}^{J}\delta_j \lambda_{m_j,r_j}(h,k)}\omega_{h,k}^2\;\rD_{h,k}\notag\\
&\quad\quad\times \exp\Bigg(\frac{\pi}{12k}(\Omega z+\Delta(h,k)z^{-1})\Bigg)\notag\\
&\quad\quad\times \prod_{j=1}^J \rZh^{\delta_j}\big(r_j\tau\gamma^*_{(m_j,h,k)}(m_j\tau)+\lambda_{m_j,r_j}(h,k)\gamma_{(m_j,h,k)}(m_j\tau);\gamma_{(m_j,h,k)}(m_j\tau)\big).\label{eq:trans-main}
\end{align}

\begin{remark}\label{rmk:Im-part}
It follows from \eqref{eq:gamma-mix} that for all $j=1,2,\ldots,J$,
\begin{align*}
0\le \Im\big(r_j\tau\gamma^*_{(m_j,h,k)}(m_j\tau)+\lambda_{m_j,r_j}(h,k)\gamma_{(m_j,h,k)}(m_j\tau)\big)<\Im\big(\gamma_{(m_j,h,k)}(m_j\tau)\big).
\end{align*}
\end{remark}

\section{Some auxiliary results}

\subsection{Necessary bounds}

Now we are going to present some useful bounds, which were obtained in the previous work; see, for example, \cite{Rad1937}.

\medskip

First, it is well known (cf.~Chapter 3 in \cite{HW1979}) that for a Farey fraction $h/k$ of order $N$, one has
\begin{equation}\label{eq:theta-bound}
\frac{1}{2kN}\le \theta'_{h,k},\theta''_{h,k}\le \frac{1}{kN}
\end{equation}
and hence
\begin{equation}\label{eq:xi-bound}
\frac{1}{kN}\le |\xi_{h,k}| \le \frac{2}{kN}.
\end{equation}

Next, since $z=k(\varrho -i \phi)$, it follows that
\begin{equation}\label{eq:Re-z-bound}
\Re(z)=k\varrho=\frac{k}{N^2}.
\end{equation}
Further, one has
\begin{align}\label{eq:Re-1-z-bound}
\Re\left(\frac{1}{z}\right)\ge \frac{k}{2}
\end{align}
since
\begin{align*}
\Re\left(\frac{1}{z}\right)=\frac{1}{k}\frac{\varrho}{\varrho^2+\phi^2}\ge \frac{1}{k}\frac{N^{-2}}{N^{-4} + k^{-2} N^{-2}}=\frac{k}{k^2N^{-2}+1}\ge \frac{k}{1+1}=\frac{k}{2},
\end{align*}
where we use the fact $k\le N$ in the last inequality.

\subsection{A partition-theoretic result}\label{sec:ptn-res}

Let $\eta$ be a positive integer. Let $p^*_{\eta}(s,t;n)$ denote the number of 2-colored (say, red and blue) partition $\eta$-tuples of $n$ with $s$ parts in total colored by red and $t$ parts in total colored by blue. Here we allow $0$ as a part. Let $q$, $\zeta$ and $\xi$ be such that $|q|<1$, $|\zeta|<1$ and $|\xi|<1$. The following infinite triple summation
$$\sum_{n\ge 0}\sum_{s\ge 0}\sum_{t\ge 0}p^*_{\eta}(s,t;n)\zeta^s \xi^t q^n=\left(\frac{1}{(\zeta,\xi;q)_\infty}\right)^{\eta}$$
is absolutely convergent. Further, considering another absolutely convergent infinite triple summation
$$\sum_{n\ge 0}\sum_{s\ge 0}\sum_{t\ge 0}d^*_{\eta}(s,t;n)\zeta^s \xi^t q^n:=(\zeta,\xi;q)_\infty^{\eta},$$
an easy partition-theoretic argument indicates that $|d^*_{\eta}(s,t;n)|\le p^*_{\eta}(s,t;n)$ for all $s,t,n\ge 0$. Also, we have $d^*_{\eta}(0,0;0)=p^*_{\eta}(0,0;0)=1$.

In general, for a nonzero integer $\delta$, if we write
$$\sum_{n\ge 0}\sum_{s\ge 0}\sum_{t\ge 0}a_{\delta}(s,t;n)\zeta^s \xi^t q^n:=(\zeta,\xi;q)_\infty^{\delta},$$
then
$$a_{\delta}(s,t;n)=\begin{cases}
p^*_{|\delta|}(s,t;n) & \text{if $\delta<0$},\\
d^*_{|\delta|}(s,t;n) & \text{if $\delta>0$},
\end{cases}$$
and hence $|a_{\delta}(s,t;n)|\le p^*_{|\delta|}(s,t;n)$ for all $s,t,n\ge 0$. Trivially, we also have
\begin{align*}
\Big|(\zeta,\xi;q)_\infty^{\delta}\Big|&=\Bigg|\sum_{n\ge 0}\sum_{s\ge 0}\sum_{t\ge 0}a_{\delta}(s,t;n)\zeta^s \xi^t q^n\Bigg|\\
&\le \sum_{n\ge 0}\sum_{s\ge 0}\sum_{t\ge 0}p^*_{|\delta|}(s,t;n)|\zeta|^s |\xi|^t |q|^n.
\end{align*}
Further, for real $0\le \alpha,\beta,x<1$, we have
\begin{align}
\sum_{n\ge 0}\sum_{s\ge 0}\sum_{t\ge 0}p^*_{1}(s,t;n)\alpha^s \beta^t x^n&=\frac{1}{(\alpha,\beta;x)_\infty}\notag\\
&=\exp\Bigg(-\sum_{k\ge 0}\log(1-\alpha x^k)-\sum_{\ell\ge 0}\log(1-\beta x^\ell)\Bigg)\notag\\
&\le \exp\Bigg(\frac{\alpha}{1-\alpha}+\frac{\alpha x}{(1-x)^2}+\frac{\beta}{1-\beta}+\frac{\beta x}{(1-x)^2}\Bigg).\label{eq:some-bound}
\end{align}

\section{Outline of the proof}

We know from \eqref{eq:cauchy-var} and \eqref{eq:trans-main} that
\begin{align*}
g(n)&=\sum_{1\le k\le N} \sum_{\substack{0\le h< k\\ \gcd(h,k)=1}} e^{-\frac{2\pi i nh}{k}} \int_{\xi_{h,k}} G\big(e^{2\pi i \tau}\big)e^{-2\pi i n \phi} e^{2 \pi n \varrho}\ d\phi\\
&=i^{\sum_{j=1}^{J}\delta_j}\sum_{1\le k\le N} \sum_{\substack{0\le h< k\\ \gcd(h,k)=1}} e^{-\frac{2\pi i nh}{k}}(-1)^{\sum_{j=1}^{J}\delta_j \lambda_{m_j,r_j}(h,k)}\omega_{h,k}^2\;\rD_{h,k}\\
&\times \int_{\xi_{h,k}} \exp\Bigg(\frac{\pi}{12k}(\Omega z+\Delta(h,k)z^{-1})\Bigg)\\
&\quad\times \prod_{j=1}^J \rZh^{\delta_j}\big(r_j\tau\gamma^*_{(m_j,h,k)}(m_j\tau)+\lambda_{m_j,r_j}(h,k)\gamma_{(m_j,h,k)}(m_j\tau);\gamma_{(m_j,h,k)}(m_j\tau)\big)\\
&\quad\times e^{-2\pi i n \phi} e^{2 \pi n \varrho}\ d\phi.
\end{align*}

Let us fix a Farey fraction $h/k$. We first find integers $1\le \ell\le L$ and $0\le \vk< \ell$ such that $k\equiv \ell \pmod{L}$ and $h\equiv \vk \pmod{\ell}$. For convenience, we write $\rho(h,k):=(\vk,\ell)$. It is not hard to observe that for all $j=1,2,\ldots,J$,
$$\gcd(m_j,k)=\gcd(m_j,\ell)\quad\text{and}\quad \lambda^*_{m_j,r_j}(h,k)=\lambda^*_{m_j,r_j}(\vk,\ell).$$
It turns out that $\Delta(h,k)=\Delta(\vk,\ell)$. We now split $g(n)$ as follows.
\begin{align*}
g(n)&=i^{\sum_{j=1}^{J}\delta_j}\sum_{1\le \ell\le L}\sum_{0\le \vk< \ell}\sum_{\substack{1\le k\le N\\k \equiv \ell \bmod{L}}} \sum_{\substack{0\le h< k\\ \gcd(h,k)=1\\ h\equiv \vk \bmod{\ell}}} e^{-\frac{2\pi i nh}{k}}\\
&\times(-1)^{\sum_{j=1}^{J}\delta_j \lambda_{m_j,r_j}(h,k)}\omega_{h,k}^2\;\rD_{h,k}\\
&\times \int_{\xi_{h,k}} \exp\Bigg(\frac{\pi}{12k}(\Omega z+\Delta(\vk,\ell)z^{-1})\Bigg)\\
&\quad\times \prod_{j=1}^J \rZh^{\delta_j}\big(r_j\tau\gamma^*_{(m_j,h,k)}(m_j\tau)+\lambda_{m_j,r_j}(h,k)\gamma_{(m_j,h,k)}(m_j\tau);\gamma_{(m_j,h,k)}(m_j\tau)\big)\\
&\quad\times e^{-2\pi i n \phi} e^{2 \pi n \varrho}\ d\phi\\
&=:i^{\sum_{j=1}^{J}\delta_j}\sum_{1\le \ell\le L}\sum_{0\le \vk< \ell} S_{\vk,\ell}.
\end{align*}

The minor arcs are those with respect to $h/k$ with $\rho(h,k)\in\mathcal{L}_{\le 0}$. We have the following bound.

\begin{theorem}\label{th:minor}
Let $(\vk,\ell)\in \mathcal{L}_{\le 0}$. For positive integers $n>-\Omega/24$, we have
$$S_{\vk,\ell}\ll_{\bfm,\bfr,\bfd} \exp\Bigg(\frac{2\pi}{N^2} \bigg(n+\frac{\Omega}{24}\bigg)\Bigg).$$
In particular, if we take $N=\myfloor{\sqrt{2\pi \left(n+\frac{\Omega}{24}\right)}}$, then $S_{\vk,\ell}\ll_{\bfm,\bfr,\bfd} 1$.
\end{theorem}

The arcs with respect to $h/k$ with $\rho(h,k)\in\mathcal{L}_{> 0}$ give us the main contribution.

\begin{theorem}\label{th:major}
Let $(\vk,\ell)\in \mathcal{L}_{> 0}$. If the inequality
\begin{align}\label{eq:condition-1}
\min_{1\le j\le J}\left(\Upsilon\big(\lambda^*_{m_j,r_j}(\vk,\ell)\big)\frac{\gcd^2(m_j,\ell)}{m_j}\right)\ge \frac{\Delta(\vk,\ell)}{24}
\end{align}
holds, then for positive integers $n>-\Omega/24$, we have
\begin{align*}
S_{\vk,\ell}&=E_{\vk,\ell}+\sum_{\substack{1\le k\le N\\k \equiv \ell \bmod{L}}} \sum_{\substack{0\le h< k\\ \gcd(h,k)=1\\ h\equiv \vk \bmod{\ell}}} e^{-\frac{2\pi i nh}{k}}(-1)^{\sum_{j=1}^{J}\delta_j \lambda_{m_j,r_j}(h,k)}\omega_{h,k}^2\;\rD_{h,k}\Pi_{h,k}\\
&\quad\quad\quad\quad\times \frac{2\pi}{k} \left(\frac{24n+\Omega}{\Delta(\vk,\ell)}\right)^{-\frac{1}{2}} I_{-1}\left(\frac{\pi }{6k}\sqrt{\Delta(\vk,\ell)(24n+\Omega)}\right),
\end{align*}
where
\begin{align*}
E_{\vk,\ell}\ll_{\bfm,\bfr,\bfd} e^{\frac{2\pi}{N^2} (n+\frac{\Omega}{24})}+\frac{N^2 e^{\frac{2\pi}{N^2} \left(n+\frac{\Omega}{24}\right)}}{n+\frac{\Omega}{24}}.
\end{align*}
In particular, if we take $N=\myfloor{\sqrt{2\pi \left(n+\frac{\Omega}{24}\right)}}$, then $E_{\vk,\ell}\ll_{\bfm,\bfr,\bfd} 1$.
\end{theorem}

Theorems \ref{th:minor} and \ref{th:major} immediately imply the main result. Before presenting proofs of the two results respectively in Sections \ref{sec:minor} and \ref{sec:major}, we make the following preparations.

\medskip

For fixed $\vk$ and $\ell$ with $1\le \ell\le L$ and $0\le \vk< \ell$, one may split the indexes $\{1,2,\ldots,J\}$ into two disjoint parts:
$$\cJ^*_{\vk,\ell}=\{j_1^*,\ldots,j_\alpha^*\} \quad\text{and}\quad \cJ^{**}_{\vk,\ell}=\{j_1^{**},\ldots,j_\beta^{**}\},$$
so that for $j^*\in \cJ^*_{\vk,\ell}$ we have $\lambda^*_{m_{j^*},r_{j^*}}(\vk,\ell)=0$ and for $j^{**}\in \cJ^{**}_{\vk,\ell}$ we have $\lambda^*_{m_{j^{**}},r_{j^{**}}}(\vk,\ell)\ne 0$.

\begin{proposition}\label{lem:pre}
Let $j^* \in \cJ^*_{\vk,\ell}$. For any Farey fraction $h/k$ such that $k\equiv \ell \pmod{L}$ and $h\equiv \vk \pmod{\ell}$, we have that
\begin{align}
&r_{j^*}\tau\gamma^*_{(m_{j^*},h,k)}(m_{j^*}\tau)+\lambda_{m_{j^*},r_{j^*}}(h,k)\gamma_{(m_{j^*},h,k)}(m_{j^*}\tau)\notag\\
&\quad=\frac{r_{j^*}\gcd(m_{j^*},k)+r_{j^*}\hbar_{m_{j^*}}(h,k)m_{j^*}h}{m_{j^*}k}
\end{align}
is a real noninteger. Further,
\begin{align}\label{eq:ineq-important}
\Big|1-e^{\frac{2\pi i}{m_{j^*}}}\Big|\le \Bigg|1-e^{2\pi i \big(r_{j^*}\tau\gamma^*_{(m_{j^*},h,k)}(m_{j^*}\tau)+\lambda_{m_{j^*},r_{j^*}}(h,k)\gamma_{(m_{j^*},h,k)}(m_{j^*}\tau)\big)}\Bigg|\le 2.
\end{align}
\end{proposition}

\begin{proof}
In this proof, we write for short $m=m_{j^*}$ and $r=r_{j^*}$. We also write $d=\gcd(m,k)$, $m=dm'$ and $k=dk'$. Since $j^* \in \cJ^*_{\vk,\ell}$, we have $\lambda^*_{m,r}(h,k)=\lambda^*_{m,r}(\vk,\ell)=0$. Hence $d$ divides $rh$ and $\lambda_{m,r}(h,k)=rh/d$. We know from \eqref{eq:gamma-mix} that
\begin{align*}
&r\tau\gamma^*_{(m,h,k)}(m\tau)+\lambda_{m,r}(h,k)\gamma_{(m,h,k)}(m\tau)\\
&\quad=\frac{r d}{mk}+\lambda_{m,r}(h,k)\frac{\hbar_m(h,k) d}{k}+\lambda^*_{m,r}(h,k)\frac{d^2}{mkz}i\\
&\quad=\frac{r d}{mk}+\lambda_{m,r}(h,k)\frac{\hbar_m(h,k) d}{k}\\
&\quad=\frac{r d}{mk}+\frac{rh}{d}\frac{\hbar_m(h,k) d}{k}\\
&\quad=\frac{r(1+\hbar_m(h,k)m'h)}{m'k}\\
&\quad=\frac{b_{m'}r}{m}=\frac{b_{m'}}{m'}\frac{r}{d},
\end{align*}
where as in Section \ref{sec:Farey}, we have put $b_{m'}=(\hbar_m(h,k) m' h+1)/k'$. Hence it is a real number.

Notice that $d=\gcd(m,k)$. Since $\gcd(h,k)=1$, $d\mid rh$ implies that $d\mid r$. Further, $b_{m'}=(\hbar_m(h,k) m' h+1)/k'$ implies that $\gcd(m',b_{m'})=1$. Hence, if $\frac{b_{m'}}{m'}\frac{r}{d}$ is an integer, then $m'\mid \frac{r}{d}$ so that $m=dm'\mid r$. This violates the assumption that $1\le r\le m-1$. Hence $r\tau\gamma^*_{(m,h,k)}(m\tau)+\lambda_{m,r}(h,k)\gamma_{(m,h,k)}(m\tau)$ is not an integer and \eqref{eq:ineq-important} follows immediately.
\end{proof}

\begin{remark}\label{rmk:Pi-value}
Recall that $\hbar_m(h,k)$ is defined to be an integer such that
$$\hbar_m(h,k) \frac{m h}{\gcd(m,k)}\equiv -1 \pmod{\frac{k}{\gcd(m,k)}}.$$
Let $n$ be an integer. It turns out that
\begin{align*}
&\exp\Bigg(2\pi i\frac{r_{j^*}\gcd(m_{j^*},k)+r_{j^*}\Big(\hbar_{m_{j^*}}(h,k)+n\frac{k}{\gcd(m_{j^*},k)}\Big)m_{j^*}h}{m_{j^*}k}\Bigg)\\
&\quad=\exp\Bigg(2\pi i\frac{r_{j^*}\gcd(m_{j^*},k)+r_{j^*}\hbar_{m_{j^*}}(h,k)m_{j^*}h}{m_{j^*}k}+2n\pi i\frac{r_{j^*}h}{\gcd(m_{j^*},k)}\Bigg)\\
&\quad=\exp\Bigg(2\pi i\frac{r_{j^*}\gcd(m_{j^*},k)+r_{j^*}\hbar_{m_{j^*}}(h,k)m_{j^*}h}{m_{j^*}k}\Bigg),
\end{align*}
since from the above proof we have $\gcd(m_{j^*},k)\mid r_{j^*}$. Hence the choice of $\hbar_{m_{j^*}}(h,k)$ does not affect the value of
\begin{align*}
&\exp\Bigg(2\pi i\bigg(r_{j^*}\tau\gamma^*_{(m_{j^*},h,k)}(m_{j^*}\tau)+\lambda_{m_{j^*},r_{j^*}}(h,k)\gamma_{(m_{j^*},h,k)}(m_{j^*}\tau)\bigg)\Bigg)\\
&\quad=\exp\Bigg(2\pi i\frac{r_{j^*}\gcd(m_{j^*},k)+r_{j^*}\hbar_{m_{j^*}}(h,k)m_{j^*}h}{m_{j^*}k}\Bigg).
\end{align*}
\end{remark}

\section{Minor arcs}\label{sec:minor}

Let $(\vk,\ell)\in\mathcal{L}_{\le 0}$, namely, $\Delta(\vk,\ell)\le 0$. We write $\cJ^*=\cJ^*_{\vk,\ell}$ and $\cJ^{**}=\cJ^{**}_{\vk,\ell}$. Notice that
\begin{align*}
|S_{\vk,\ell}|&\le \sum_{\substack{1\le k\le N\\k \equiv \ell \bmod{L}}} \sum_{\substack{0\le h< k\\ \gcd(h,k)=1\\ h\equiv \vk \bmod{\ell}}} \int_{\xi_{h,k}} \exp\Bigg(\frac{\pi}{12k}(\Omega \Re(z)+\Delta(\vk,\ell)\Re(z^{-1}))\Bigg)\\
&\times \Bigg|\prod_{j=1}^J \rZh^{\delta_j}\big(r_j\tau\gamma^*_{(m_j,h,k)}(m_j\tau)+\lambda_{m_j,r_j}(h,k)\gamma_{(m_j,h,k)}(m_j\tau);\gamma_{(m_j,h,k)}(m_j\tau)\big)\Bigg|\\
&\times e^{2 \pi n \varrho}\ d\phi.
\end{align*}

We now consider the Farey arcs with respect to $h/k$ with $k\equiv \ell \pmod{L}$ and $h\equiv \vk \pmod{\ell}$. Since $\Delta(\vk,\ell)\le 0$, it follows from \eqref{eq:Re-z-bound} and \eqref{eq:Re-1-z-bound} that
\begin{align*}
\exp\Bigg(\frac{\pi}{12k}(\Omega \Re(z)+\Delta(\vk,\ell)\Re(z^{-1}))\Bigg)&\le \exp\Bigg(\frac{\pi}{12k}\Big(\Omega \frac{k}{N^2}+\Delta(\vk,\ell)\frac{k}{2}\Big)\Bigg)\\
&=\exp\Bigg(\frac{\pi \varrho\; \Omega}{12}\Bigg)\exp\Bigg(\frac{\pi \Delta(\vk,\ell)}{24}\Bigg).
\end{align*}

For convenience, now we write $\lambda_j=\lambda_{m_j,r_j}(h,k)$ and $\lambda^*_j=\lambda^*_{m_j,r_j}(h,k)$. We also write for short $\tilde{\vs}_j=r_j\tau\gamma^*_{(m_j,h,k)}(m_j\tau)+\lambda_{m_j,r_j}(h,k)\gamma_{(m_j,h,k)}(m_j\tau)$ and $\tilde{\tau}_j=\gamma_{(m_j,h,k)}(m_j\tau)$. We know from \eqref{eq:gamma} and \eqref{eq:gamma-mix} that
$$\Im(\tilde{\tau}_j)=\frac{\gcd^2(m_j,k)}{m_jk}\Re(z^{-1})=\frac{\gcd^2(m_j,\ell)}{m_jk}\Re(z^{-1})$$
and
$$\Im(\tilde{\vs}_j)=\lambda^*_j\frac{\gcd^2(m_j,k)}{m_jk}\Re(z^{-1})=\lambda^*_j\frac{\gcd^2(m_j,\ell)}{m_jk}\Re(z^{-1}).$$
Notice that
$$0\le \Im(\tilde{\vs}_j)<\Im(\tilde{\tau}_j).$$
We write
\begin{align*}
\prod_{j=1}^J \rZh^{\delta_j}(\tilde{\vs}_j;\tilde{\tau}_j)&=\prod_{j^*\in\cJ^*}(1-e^{2\pi i \tilde{\vs}_{j^*}})^{\delta_{j^*}}\\
&\quad\times \prod_{j^*\in\cJ^*}(e^{2\pi i (\tilde{\tau}_{j^*}+\tilde{\vs}_{j^*})},e^{2\pi i (\tilde{\tau}_{j^*}-\tilde{\vs}_{j^*})};e^{2\pi i \tilde{\tau}_{j^*}})_\infty^{\delta_{j^*}}\\
&\quad\times \prod_{j^{**}\in\cJ^{**}}(e^{2\pi i \tilde{\vs}_{j^{**}}},e^{2\pi i (\tilde{\tau}_{j^{**}}-\tilde{\vs}_{j^{**}})};e^{2\pi i \tilde{\tau}_{j^{**}}})_\infty^{\delta_{j^{**}}}.
\end{align*}
First, it follows from Proposition \ref{lem:pre} that
$$\prod_{j^*\in\cJ^*}(1-e^{2\pi i \tilde{\vs}_{j^*}})^{\delta_{j^*}}\ll 1.$$
Further, as we have seen in Section \ref{sec:ptn-res}, for $j^*\in\cJ^*$ (hence $\lambda^*_{j^*}=0$),
\begin{align*}
&\Big|(e^{2\pi i (\tilde{\tau}_{j^*}+\tilde{\vs}_{j^*})},e^{2\pi i (\tilde{\tau}_{j^*}-\tilde{\vs}_{j^*})};e^{2\pi i \tilde{\tau}_{j^*}})_\infty^{\delta_{j^*}}\Big|\\
& \le \sum_{n\ge 0}\sum_{s\ge 0}\sum_{t\ge 0}p^*_{|\delta_{j^*}|}(s,t;n)|e^{2\pi i (\tilde{\tau}_{j^*}+\tilde{\vs}_{j^*})}|^s |e^{2\pi i (\tilde{\tau}_{j^*}-\tilde{\vs}_{j^*})}|^t |e^{2\pi i \tilde{\tau}_{j^*}}|^n\\
&=\sum_{n\ge 0}\sum_{s\ge 0}\sum_{t\ge 0}p^*_{|\delta_{j^*}|}(s,t;n) e^{-2\pi \Im(\tilde{\tau}_{j^*}+\tilde{\vs}_{j^*})s}e^{-2\pi \Im(\tilde{\tau}_{j^*}-\tilde{\vs}_{j^*})t}e^{-2\pi \Im(\tilde{\tau}_{j^*})n}\\
&=\sum_{n\ge 0}\sum_{s\ge 0}\sum_{t\ge 0}p^*_{|\delta_{j^*}|}(s,t;n) \exp\bigg(-2\pi \frac{\gcd^2(m_{j^*},\ell)}{m_{j^*}k}\Re(z^{-1})s\bigg)\\
&\quad\times\exp\bigg(-2\pi \frac{\gcd^2(m_{j^*},\ell)}{m_{j^*}k}\Re(z^{-1})t\bigg)\exp\bigg(-2\pi \frac{\gcd^2(m_{j^*},\ell)}{m_{j^*}k}\Re(z^{-1})n\bigg)\\
&\le \sum_{n\ge 0}\sum_{s\ge 0}\sum_{t\ge 0}p^*_{|\delta_{j^*}|}(s,t;n) \exp\bigg(-\pi \frac{\gcd^2(m_{j^*},\ell)}{m_{j^*}}s\bigg)\\
&\quad\times\exp\bigg(-\pi \frac{\gcd^2(m_{j^*},\ell)}{m_{j^*}}t\bigg)\exp\bigg(-\pi \frac{\gcd^2(m_{j^*},\ell)}{m_{j^*}}n\bigg),
\end{align*}
where we use $\Re(z^{-1})\ge k/2$. It follows from \eqref{eq:some-bound} that
$$(e^{2\pi i (\tilde{\tau}_{j^*}+\tilde{\vs}_{j^*})},e^{2\pi i (\tilde{\tau}_{j^*}-\tilde{\vs}_{j^*})};e^{2\pi i \tilde{\tau}_{j^*}})_\infty^{\delta_{j^*}}\ll 1.$$
Likewise, for $j^{**}\in\cJ^{**}$,
\begin{align*}
&\Big|(e^{2\pi i \tilde{\vs}_{j^{**}}},e^{2\pi i (\tilde{\tau}_{j^{**}}-\tilde{\vs}_{j^{**}})};e^{2\pi i \tilde{\tau}_{j^{**}}})_\infty^{\delta_{j^{**}}}\Big|\\
&\quad\le \sum_{n\ge 0}\sum_{s\ge 0}\sum_{t\ge 0}p^*_{|\delta_{j^{**}}|}(s,t;n) \exp\bigg(-\pi \lambda^*_{j^{**}}\frac{\gcd^2(m_{j^{**}},\ell)}{m_{j^{**}}}s\bigg)\\
&\quad\quad\times\exp\bigg(-\pi (1-\lambda^*_{j^{**}})\frac{\gcd^2(m_{j^{**}},\ell)}{m_{j^{**}}}t\bigg)\exp\bigg(-\pi \frac{\gcd^2(m_{j^{**}},\ell)}{m_{j^{**}}}n\bigg)\\
&\quad\ll 1.
\end{align*}

Hence,
\begin{align*}
S_{\vk,\ell}&\ll \sum_{\substack{1\le k\le N\\k \equiv \ell \bmod{L}}} \sum_{\substack{0\le h< k\\ \gcd(h,k)=1\\ h\equiv \vk \bmod{\ell}}} \int_{\xi_{h,k}}  e^{\frac{\pi \varrho\; \Omega}{12}} e^{2 \pi n \varrho}\ d\phi\\
&\ll \sum_{\substack{1\le k\le N\\k \equiv \ell \bmod{L}}} \sum_{\substack{0\le h< k\\ \gcd(h,k)=1\\ h\equiv \vk \bmod{\ell}}}e^{2\pi \varrho(n+\frac{\Omega}{24})} \frac{1}{kN}\\
&\ll e^{2\pi \varrho(n+\frac{\Omega}{24})}=e^{\frac{2\pi}{N^2} (n+\frac{\Omega}{24})}.
\end{align*}

\section{Major arcs}\label{sec:major}

Let $(\vk,\ell)\in\mathcal{L}_{> 0}$, namely, $\Delta(\vk,\ell)> 0$. Again, we write $\cJ^*=\cJ^*_{\vk,\ell}$ and $\cJ^{**}=\cJ^{**}_{\vk,\ell}$. Let us consider the Farey arcs with respect to $h/k$ with $k\equiv \ell \pmod{L}$ and $h\equiv \vk \pmod{\ell}$. For convenience, we write $\tilde{\vs}_j(h,k)=r_j\tau\gamma^*_{(m_j,h,k)}(m_j\tau)+\lambda_{m_j,r_j}(h,k)\gamma_{(m_j,h,k)}(m_j\tau)$ and $\tilde{\tau}_j(h,k)=\gamma_{(m_j,h,k)}(m_j\tau)$.

\smallskip

Recall that
\begin{align*}
S_{\vk,\ell}&=\sum_{\substack{1\le k\le N\\k \equiv \ell \bmod{L}}} \sum_{\substack{0\le h< k\\ \gcd(h,k)=1\\ h\equiv \vk \bmod{\ell}}} e^{-\frac{2\pi i nh}{k}}(-1)^{\sum_{j=1}^{J}\delta_j \lambda_{m_j,r_j}(h,k)}\omega_{h,k}^2\;\rD_{h,k}\\
&\times \int_{\xi_{h,k}} \exp\Bigg(\frac{\pi}{12k}(\Omega z+\Delta(\vk,\ell)z^{-1})\Bigg) \prod_{j=1}^J \rZh^{\delta_j}\big(\tilde{\vs}_j(h,k);\tilde{\tau}_j(h,k)\big)\\
&\times e^{-2\pi i n \phi} e^{2 \pi n \varrho}\ d\phi.
\end{align*}
We split $S_{\vk,\ell}$ into two parts $\Sigma_1$ and $\Sigma_2$ where
\begin{align*}
\Sigma_1&:=\sum_{\substack{1\le k\le N\\k \equiv \ell \bmod{L}}} \sum_{\substack{0\le h< k\\ \gcd(h,k)=1\\ h\equiv \vk \bmod{\ell}}} e^{-\frac{2\pi i nh}{k}}(-1)^{\sum_{j=1}^{J}\delta_j \lambda_{m_j,r_j}(h,k)}\omega_{h,k}^2\;\rD_{h,k}\\
&\;\times \int_{\xi_{h,k}} \exp\Bigg(\frac{\pi}{12k}(\Omega z+\Delta(\vk,\ell)z^{-1})\Bigg) \Pi_{h,k} e^{-2\pi i n \phi} e^{2 \pi n \varrho}\ d\phi
\end{align*}
and
\begin{align*}
\Sigma_2&:=\sum_{\substack{1\le k\le N\\k \equiv \ell \bmod{L}}} \sum_{\substack{0\le h< k\\ \gcd(h,k)=1\\ h\equiv \vk \bmod{\ell}}} e^{-\frac{2\pi i nh}{k}}(-1)^{\sum_{j=1}^{J}\delta_j \lambda_{m_j,r_j}(h,k)}\omega_{h,k}^2\;\rD_{h,k}\\
&\;\times \int_{\xi_{h,k}} \exp\Bigg(\frac{\pi}{12k}(\Omega z+\Delta(\vk,\ell)z^{-1})\Bigg) \Bigg(\prod_{j=1}^J \rZh^{\delta_j}\big(\tilde{\vs}_j(h,k);\tilde{\tau}_j(h,k)\big)-\Pi_{h,k}\Bigg)\\
&\;\times e^{-2\pi i n \phi} e^{2 \pi n \varrho}\ d\phi.
\end{align*}

We first show that $\Sigma_2$ is negligible. Notice that by \eqref{eq:Re-z-bound}
\begin{align*}
|\Sigma_2|&\le\sum_{\substack{1\le k\le N\\k \equiv \ell \bmod{L}}} \sum_{\substack{0\le h< k\\ \gcd(h,k)=1\\ h\equiv \vk \bmod{\ell}}} e^{2\pi \varrho(n+\frac{\Omega}{24})}|\Pi_{h,k}|\\
&\times \int_{\xi_{h,k}} \exp\Bigg(\frac{\pi\Delta(\vk,\ell)}{12k}\Re(z^{-1})\Bigg) \Bigg|\frac{1}{\Pi_{h,k}}\prod_{j=1}^J \rZh^{\delta_j}\big(\tilde{\vs}_j(h,k);\tilde{\tau}_j(h,k)\big)-1\Bigg|\ d\phi.
\end{align*}
Let us fix $h$ and $k$ and write $\tilde{\vs}_j=\tilde{\vs}_j(h,k)$ and $\tilde{\tau}_j=\tilde{\tau}_j(h,k)$. We also write $\lambda^*_j=\lambda^*_{m_j,r_j}(h,k)$. Recalling the definition of $\Pi_{h,k}$ and Proposition \ref{lem:pre}, we have
\begin{align*}
\frac{1}{\Pi_{h,k}}\prod_{j=1}^J \rZh^{\delta_j}\big(\tilde{\vs}_j;\tilde{\tau}_j\big)-1&=\prod_{j^*\in\cJ^*}(e^{2\pi i (\tilde{\tau}_{j^*}+\tilde{\vs}_{j^*})},e^{2\pi i (\tilde{\tau}_{j^*}-\tilde{\vs}_{j^*})};e^{2\pi i \tilde{\tau}_{j^*}})_\infty^{\delta_{j^*}}\\
&\quad\times \prod_{j^{**}\in\cJ^{**}}(e^{2\pi i \tilde{\vs}_{j^{**}}},e^{2\pi i (\tilde{\tau}_{j^{**}}-\tilde{\vs}_{j^{**}})};e^{2\pi i \tilde{\tau}_{j^{**}}})_\infty^{\delta_{j^{**}}}-1.
\end{align*}
Let us write for short
$$\tilde{\vs}_j^{\textrm{New}}=\begin{cases}
\tilde{\tau}_{j}+\tilde{\vs}_{j} & \text{if $j\in\cJ^*$},\\
\tilde{\vs}_{j} & \text{if $j\in\cJ^{**}$}.
\end{cases}$$
It follows again from \eqref{eq:gamma} and \eqref{eq:gamma-mix} that
$$\Im(\tilde{\tau}_j)=\frac{\gcd^2(m_j,\ell)}{m_jk}\Re(z^{-1}),$$
$$\Im(\tilde{\vs}_j)=\lambda^*_j\frac{\gcd^2(m_j,\ell)}{m_jk}\Re(z^{-1})$$
and
$$\Im(\tilde{\vs}_j^{\textrm{New}})=\Phi(\lambda^*_j)\frac{\gcd^2(m_j,\ell)}{m_jk}\Re(z^{-1}),$$
where for real $0\le x<1$,
$$\Phi(x):=\begin{cases}
1 & \text{if $x=0$},\\
x & \text{otherwise}.
\end{cases}$$
We have
\begin{align*}
&\Bigg|\frac{1}{\Pi_{h,k}}\prod_{j=1}^J \rZh^{\delta_j}\big(\tilde{\vs}_j;\tilde{\tau}_j\big)-1\Bigg|\\
&\quad=\Bigg|\prod_{j=1}^J(e^{2\pi i \tilde{\vs}_j^{\textrm{New}}},e^{2\pi i (\tilde{\tau}_{j}-\tilde{\vs}_{j})};e^{2\pi i \tilde{\tau}_{j}})_\infty^{\delta_{j}}-1\Bigg|\\
&\quad\le \sum_{\mathbf{n}:=(n_1,\ldots,n_J)\in\mathbb{Z}_{\ge 0}^J}\sum_{\mathbf{s}:=(s_1,\ldots,s_J)\in\mathbb{Z}_{\ge 0}^J}\sum_{\mathbf{t}:=(t_1,\ldots,t_J)\in\mathbb{Z}_{\ge 0}^J}\\
&\quad\quad\prod_{j=1}^J p^*_{|\delta_{j}|}(s_j,t_j;n_j)|e^{2\pi i \tilde{\vs}_j^{\textrm{New}}}|^{s_j} |e^{2\pi i (\tilde{\tau}_{j}-\tilde{\vs}_{j})}|^{t_j} |e^{2\pi i \tilde{\tau}_{j}}|^{n_j}-1\\
&\quad=\underset{\mathbf{n}\times\mathbf{s}\times\mathbf{t}\in (\mathbb{Z}_{\ge 0}^J)^3\backslash (0,\ldots,0)^3}{\sum\sum\sum}\prod_{j=1}^J p^*_{|\delta_{j}|}(s_j,t_j;n_j)|e^{2\pi i \tilde{\vs}_j^{\textrm{New}}}|^{s_j} |e^{2\pi i (\tilde{\tau}_{j}-\tilde{\vs}_{j})}|^{t_j} |e^{2\pi i \tilde{\tau}_{j}}|^{n_j}\\
&\quad=\underset{\mathbf{n}\times\mathbf{s}\times\mathbf{t}\in(\mathbb{Z}_{\ge 0}^J)^3\backslash (0,\ldots,0)^3}{\sum\sum\sum}\prod_{j=1}^J p^*_{|\delta_{j}|}(s_j,t_j;n_j) e^{-2\pi \Im(\tilde{\vs}_j^{\textrm{New}})s_j} e^{-2\pi \Im(\tilde{\tau}_{j}-\tilde{\vs}_{j})t_j} e^{-2\pi \Im(\tilde{\tau}_{j})n_j}\\
&\quad=\underset{\mathbf{n}\times\mathbf{s}\times\mathbf{t}\in(\mathbb{Z}_{\ge 0}^J)^3\backslash (0,\ldots,0)^3}{\sum\sum\sum}\Bigg(\prod_{j=1}^J p^*_{|\delta_{j}|}(s_j,t_j;n_j)\Bigg) \\
&\quad\quad\times \exp\Bigg(-2\pi\frac{\Re(z^{-1})}{k}\sum_{j=1}^J\frac{\gcd^2(m_j,\ell)}{m_j}\big(\Phi(\lambda^*_j)s_j+(1-\lambda^*_j)t_j+n_j\big)\Bigg).
\end{align*}
Hence,
\begin{align*}
&\exp\Bigg(\frac{\pi\Delta(\vk,\ell)}{12k}\Re(z^{-1})\Bigg)\Bigg|\frac{1}{\Pi_{h,k}}\prod_{j=1}^J \rZh^{\delta_j}\big(\tilde{\vs}_j;\tilde{\tau}_j\big)-1\Bigg|\\
&\le \underset{\mathbf{n}\times\mathbf{s}\times\mathbf{t}\in(\mathbb{Z}_{\ge 0}^J)^3\backslash (0,\ldots,0)^3}{\sum\sum\sum}\Bigg(\prod_{j=1}^J p^*_{|\delta_{j}|}(s_j,t_j;n_j)\Bigg) \\
&\times \exp\Bigg(\!\!-2\pi\frac{\Re(z^{-1})}{k}\bigg(\!\!-\frac{\Delta(\vk,\ell)}{24}+\sum_{j=1}^J\frac{\gcd^2(m_j,\ell)}{m_j}\big(\Phi(\lambda^*_j)s_j+(1-\lambda^*_j)t_j+n_j\big)\bigg)\!\Bigg).
\end{align*}
Since at least one coordinate of $\mathbf{n}\times\mathbf{s}\times\mathbf{t}$ is nonzero, under the condition \eqref{eq:condition-1}, we know that
\begin{align*}
&-\frac{\Delta(\vk,\ell)}{24}+\sum_{j=1}^J\frac{\gcd^2(m_j,\ell)}{m_j}\big(\Phi(\lambda^*_j)s_j+(1-\lambda^*_j)t_j+n_j\big)\\
&\qquad\ge -\frac{\Delta(\vk,\ell)}{24}+\min_{1\le j\le J}\left(\Upsilon(\lambda^*_j)\frac{\gcd^2(m_j,\ell)}{m_j}\right)\ge 0
\end{align*}
for all $\mathbf{n}\times\mathbf{s}\times\mathbf{t}\in(\mathbb{Z}_{\ge 0}^J)^3\backslash (0,\ldots,0)^3$. Recalling that $\Re(z^{-1})\ge k/2$, it follows that
$$\exp\Bigg(\frac{\pi\Delta(\vk,\ell)}{12k}\Re(z^{-1})\Bigg)\Bigg|\frac{1}{\Pi_{h,k}}\prod_{j=1}^J \rZh^{\delta_j}\big(\tilde{\vs}_j;\tilde{\tau}_j\big)-1\Bigg|$$
is maximized when $\Re(z^{-1})=k/2$. Namely,
\begin{align*}
&\exp\Bigg(\frac{\pi\Delta(\vk,\ell)}{12k}\Re(z^{-1})\Bigg)\Bigg|\frac{1}{\Pi_{h,k}}\prod_{j=1}^J \rZh^{\delta_j}\big(\tilde{\vs}_j;\tilde{\tau}_j\big)-1\Bigg|\\
&\quad\le \exp\Bigg(\frac{\pi\Delta(\vk,\ell)}{24}\Bigg)\underset{\mathbf{n}\times\mathbf{s}\times\mathbf{t}\in(\mathbb{Z}_{\ge 0}^J)^3\backslash (0,\ldots,0)^3}{\sum\sum\sum}\Bigg(\prod_{j=1}^J p^*_{|\delta_{j}|}(s_j,t_j;n_j)\Bigg) \\
&\quad\quad\times \exp\Bigg(-\pi\sum_{j=1}^J\frac{\gcd^2(m_j,\ell)}{m_j}\big(\Phi(\lambda^*_j)s_j+(1-\lambda^*_j)t_j+n_j\big)\Bigg)\\
&\quad\ll 1.
\end{align*}
Together with the fact $\prod_{h,k}\ll 1$ which follows from \eqref{eq:ineq-important}, we conclude that
\begin{align*}
\Sigma_2&\ll\sum_{\substack{1\le k\le N\\k \equiv \ell \bmod{L}}} \sum_{\substack{0\le h< k\\ \gcd(h,k)=1\\ h\equiv \vk \bmod{\ell}}} e^{2\pi \varrho(n+\frac{\Omega}{24})} \int_{\xi_{h,k}} 1\ d\phi\\
&\ll \sum_{\substack{1\le k\le N\\k \equiv \ell \bmod{L}}} \sum_{\substack{0\le h< k\\ \gcd(h,k)=1\\ h\equiv \vk \bmod{\ell}}}e^{2\pi \varrho(n+\frac{\Omega}{24})} \frac{1}{kN}\\
&\ll e^{2\pi \varrho(n+\frac{\Omega}{24})}=e^{\frac{2\pi}{N^2} (n+\frac{\Omega}{24})}.
\end{align*}

Finally, we estimate the main contribution $\Sigma_1$. To do so, we need the following evaluation of an integral, which is a special case of Lemma 2.4 in \cite{Che2019}. For the sake of completeness, we sketch a brief proof.

\begin{lemma}\label{le:key-int}
Let $a\in \mathbb{R}_{>0}$ and $b\in\mathbb{R}$. Let $\gcd(h,k)=1$. Define
\begin{equation}
I:=\int_{\xi_{h,k}}e^{\frac{\pi}{12k}\left(\frac{a}{z}+bz\right)}e^{-2\pi i n\phi}e^{2\pi n \varrho}\ d\phi.
\end{equation}
Then for positive integers $n$ with $n> -b/24$, we have
\begin{equation}
I=\frac{2\pi}{k} \left(\frac{24n+b}{a}\right)^{-\frac{1}{2}} I_{-1}\left(\frac{\pi }{6k}\sqrt{a(24n+b)}\right)+E(I),
\end{equation}
where
\begin{equation}\label{eq:integral-error}
E(I)\ll_a \frac{e^{2\pi\varrho \left(n+\frac{b}{24}\right)}}{n+\frac{b}{24}}.
\end{equation}
\end{lemma}

\begin{proof}
Putting $w=z/k=\varrho-i\phi$ and reversing the integral order, one has
$$I=\frac{1}{2\pi i}\int_{\varrho-i\theta''_{h,k}}^{\varrho+i\theta'_{h,k}} 2\pi e^{\frac{\pi a}{12k^2 w}} e^{2\pi w \left(n+\frac{b}{24}\right)}\ dw.$$
We now split the integral into three parts:
\begin{align*}
I&=\frac{1}{2\pi i}\left(\int_\Gamma-\int_{-\infty-i\theta''_{h,k}}^{\varrho-i\theta''_{h,k}}+\int_{-\infty+i\theta'_{h,k}}^{\varrho+i\theta'_{h,k}}\right) 2\pi e^{\frac{\pi a}{12k^2 w}} e^{2\pi w \left(n+\frac{b}{24}\right)}\ dw\\
&=:J_1-J_2+J_3,
\end{align*}
where
\begin{align*}
\Gamma:=(-\infty-i\theta''_{h,k}) \to (\varrho-i\theta''_{h,k}) \to (\varrho+i\theta'_{h,k}) \to (-\infty+i\theta'_{h,k})
\end{align*}
is a Hankel contour.

\smallskip

The dominant contribution to $I$ comes from $J_1$. We make the following change of variables $t=wk\sqrt{(24n+b)/a}$. Then
$$J_1=\frac{2\pi}{k}\left(\frac{24n+b}{a}\right)^{-\frac{1}{2}} \frac{1}{2\pi i}\int_{\tilde{\Gamma}} e^{\frac{\pi }{12k}\sqrt{a(24n+b)}\left(t+\frac{1}{t}\right)} dt,$$
in which the new contour $\tilde{\Gamma}$ is still a Hankel contour. Recalling the contour integral representation of $I_s(x)$:
$$I_s(x)=\frac{1}{2\pi i}\int_{\Gamma} t^{-s-1}e^{\frac{x}{2}\left(t+\frac{1}{t}\right)}\ dt\quad\text{($\Gamma$ is a Hankel contour)},$$
we conclude that
$$J_1=\frac{2\pi}{k} \left(\frac{24n+b}{a}\right)^{-\frac{1}{2}} I_{-1}\left(\frac{\pi }{6k}\sqrt{a(24n+b)}\right).$$

\smallskip

We next bound the error term $E(I)$, coming from $J_2$ and $J_3$. Let us put $w=x+i\theta$ with $-\infty\le x\le \varrho$ and $\theta\in\{\theta'_{h,k},-\theta''_{h,k}\}$. We know that
$$\left|e^{2\pi w \left(n+\frac{b}{24}\right)}\right|= e^{2\pi x \left(n+\frac{b}{24}\right)}$$
and
\begin{align*}
\left|e^{\frac{\pi a}{12k^2 w}}\right|&=e^{\frac{\pi a}{12k^2}\Re\left(\frac{1}{w}\right)}=e^{\frac{\pi a}{12k^2}\frac{x}{x^2+\theta^2}}\le e^{\frac{\pi a}{12k^2}\frac{x}{\theta^2}}\le e^{\frac{\pi a}{12k^2}\varrho (2kN)^2}=e^{\frac{\pi a}{3}}\ll_a 1,
\end{align*}
where we use the fact $\frac{1}{2kN}\le |\theta|\le \frac{1}{kN}$. Hence for $j=2$ and $3$, we have
\begin{align*}
|J_j|\ll_a \int_{-\infty}^{\varrho} e^{2\pi x \left(n+\frac{b}{24}\right)}\ dx\ll_a \frac{e^{2\pi\varrho \left(n+\frac{b}{24}\right)}}{n+\frac{b}{24}}.
\end{align*}
This implies that
$$|E(I)|=|-J_2+J_3|\le |J_2|+|J_3|\ll_a \frac{e^{2\pi\varrho \left(n+\frac{b}{24}\right)}}{n+\frac{b}{24}},$$
which gives \eqref{eq:integral-error}.
\end{proof}

Recall that
\begin{align*}
\Sigma_1&=\sum_{\substack{1\le k\le N\\k \equiv \ell \bmod{L}}} \sum_{\substack{0\le h< k\\ \gcd(h,k)=1\\ h\equiv \vk \bmod{\ell}}} e^{-\frac{2\pi i nh}{k}}(-1)^{\sum_{j=1}^{J}\delta_j \lambda_{m_j,r_j}(h,k)}\omega_{h,k}^2\;\rD_{h,k}\Pi_{h,k}\\
&\quad\times \int_{\xi_{h,k}} \exp\Bigg(\frac{\pi}{12k}(\Omega z+\Delta(\vk,\ell)z^{-1})\Bigg) e^{-2\pi i n \phi} e^{2 \pi n \varrho}\ d\phi.
\end{align*}
The main contribution to $\Sigma_1$ is
\begin{align*}
&\sum_{\substack{1\le k\le N\\k \equiv \ell \bmod{L}}} \sum_{\substack{0\le h< k\\ \gcd(h,k)=1\\ h\equiv \vk \bmod{\ell}}} e^{-\frac{2\pi i nh}{k}}(-1)^{\sum_{j=1}^{J}\delta_j \lambda_{m_j,r_j}(h,k)}\omega_{h,k}^2\;\rD_{h,k}\Pi_{h,k}\\
&\times \frac{2\pi}{k} \left(\frac{24n+\Omega}{\Delta(\vk,\ell)}\right)^{-\frac{1}{2}} I_{-1}\left(\frac{\pi }{6k}\sqrt{\Delta(\vk,\ell)(24n+\Omega)}\right).
\end{align*}
The error term in $\Sigma_1$ is bounded by
\begin{align*}
\sum_{\substack{1\le k\le N\\k \equiv \ell \bmod{L}}} \sum_{\substack{0\le h< k\\ \gcd(h,k)=1\\ h\equiv \vk \bmod{\ell}}} \frac{e^{2\pi\varrho \left(n+\frac{\Omega}{24}\right)}}{n+\frac{\Omega}{24}}\ll \frac{N^2 e^{\frac{2\pi}{N^2} \left(n+\frac{\Omega}{24}\right)}}{n+\frac{\Omega}{24}}.
\end{align*}

\subsection*{Acknowledgements}

I want to thank George Andrews for helpful suggestions and for pointing out a number of useful references.

\bibliographystyle{amsplain}

\end{document}